\documentclass[12pt,a4paper,reqno]{amsart}
\allowdisplaybreaks
\usepackage{amsmath}
\usepackage{amsfonts}
\usepackage{amssymb,amsthm,amsfonts,amsthm,latexsym,enumerate,url,cases}
\numberwithin{equation}{section}
\usepackage{mathrsfs}
\usepackage{hyperref}
\hypersetup{colorlinks=true,citecolor=blue,linkcolor=blue,urlcolor=blue}
\def\pmod #1{\ ({\rm{mod}}\ #1)}

\theoremstyle{plain}
\newtheorem{theorem}{Theorem}
\newtheorem{lemma}{Lemma}

\newtheorem{proposition}{Proposition}
\newtheorem{conjecture}{Conjecture}
\theoremstyle{definition}

\usepackage{etoolbox}
\makeatletter
\patchcmd{\@settitle}{\uppercasenonmath\@title}{}{}{}
\patchcmd{\@setauthors}{\MakeUppercase}{}{}{}
\patchcmd{\section}{\scshape}{}{}{}
\makeatother

\begin{document}

\title
[{On a conjecture of M. R. Murty and V. K. Murty II}]
{On a conjecture of M. R. Murty and V. K. Murty II}

\author
[Y. Ding,\quad  V.Z. Guo\quad  {\it and}\quad Yu Zhang] {Yuchen Ding,\quad  Victor Zhenyu Guo\quad  {\it and}\quad Yu Zhang}

\address{(Yuchen Ding) School of Mathematical Science,  Yangzhou University, Yangzhou 225002, People's Republic of China}
\email{ycding@yzu.edu.cn}
\address{(Victor Zhenyu Guo) School of Mathematics and Statistics,  Xi'an Jiaotong University, Xi'an 710049, People's Republic of China}
\email{guozyv@xjtu.edu.cn}
\address{(Yu Zhang) School of Mathematics, Shandong University, Jinan 250100, Shandong, People's Republic of China}
\email{yuzhang0615@mail.sdu.edu.cn}


\begin{abstract}
Let $\omega^*(n)$ be the number of primes $p$ such that $p-1$ divides $n$. Assuming the Elliott--Halberstam Conjecture, we prove a conjecture posted by M. R.  Murty and V. K. Murty in 2021 which states that
$$\sum_{n\leqslant x}\omega^*(n)^2\sim 2\frac{\zeta(2)\zeta(3)}{\zeta(6)}x\log x, \quad \text{as} \quad x\rightarrow \infty.$$
The above sum was first investigated by Prachar in 1955.

One of the key ingredients in our argument is the application of a sieve result on estimating various certain summations involving primes in arithmetic progressions, rather than a direct use of the Brun--Titchmarsh inequality which would not be applicable for our task.
\end{abstract}
\maketitle

\section{Introduction}
The investigations of the normal order of certain arithmetic functions started from the paper of Hardy and Ramanujan \cite{HR}, proving that almost all integers $n$ satisfy $\omega(n)\sim \log\log n$ as $n\rightarrow \infty$, where $\omega(n)$ is the number of distinct prime divisors of $n$.  Later, Tur\'an \cite{Tu} simplified the proof significantly, by showing the following two asymptotic formulae:
$$\sum_{n\leqslant x}\omega(n)=x\log\log x+Bx+O(x/\log x)$$
and
$$\sum_{n\leqslant x}\omega(n)^2=x(\log\log x)^2+O(x\log\log x),$$
where $B$ is an absolute constant.

Let $\omega^*(n)$ be the number of primes $p$ such that $p-1$ divides $n$, which is a variant of the arithmetic function $\omega$. N\"{o}bauer \cite{Nobauer} studied this arithmetic function while investigating the number of a special group. Later, Prachar \cite{Pr} considered the arithmetic properties of $\omega^*$ similar to Tur\'an \cite{Tu}'s contributions on $\omega$ and proved that
$$\sum_{n\leqslant x}\omega^*(n)=x\log\log x+Bx+O(x/\log x)$$
as well as
$$\sum_{n\leqslant x}\omega^*(n)^2=O\left(x(\log x)^2\right).$$
Prachar moreover showed that $$\omega^*(n)>\exp\left(a\log n/(\log\log n)^2\right)$$
for infinitely many integers $n$, where $a$ is an absolute constant. Adleman, Pomerance and Rumely \cite{APR} improved this to $$\omega^*(n)>\exp\left(a\log n/\log\log n\right).$$
Prachar's work relies on the estimation of a function counting the number of pairs of primes $p$ and $q$ so that the least common multiple $[p-1,q-1]\leqslant x$. Erd\H os and Prachar \cite{EP} proved that the counting function is bounded by $O(x\log\log x)$ and claimed that the estimation can be improved in a remark of their article. Murty and Murty \cite{MM} followed the remark and improved the bound to $O(x)$. As a consequence, they showed that
$$x(\log\log x)^3\ll\sum_{n\leqslant x}\omega^*(n)^2\ll x\log x.$$
As remarked by Murty and Murty, the above lower bound means that $\omega^*(n)$ does not have a normal order. Moreover, Murty and Murty conjectured that there is some positive constant $C$ such that
$$\sum_{n\leqslant x}\omega^*(n)^2\sim Cx\log x$$
as $x\rightarrow \infty$, or equivalently,
$$\sum_{[p-1,q-1]\leqslant x}\frac{1}{[p-1,q-1]}\sim C\log x.$$
In a former note, Ding \cite{Di} showed that
$$a_1x\log x\leqslant \sum_{n\leqslant x}\omega^*(n)^2\leqslant a_2x\log x,$$
where $a_1$ and $a_2$ are two absolute constants.

In this article, we give a conditional proof of the conjecture by Murty and Murty. Precisely, we confirm the conjecture under the Elliott--Halberstam Conjecture  \cite{EH}, which is of the following.

\begin{conjecture}[Elliott--Halberstam]\label{conjecture1} For any fixed positive real numbers $A$ and $\epsilon$, we have the following estimate
$$\sum_{d\leqslant x^{1-\epsilon}}\max_{y\leqslant x}\left|\pi(y;d,1)-\frac{\mathrm{li}y}{\varphi(d)}\right|\ll_{A,\epsilon} \frac{x}{(\log x)^A},$$
where $\pi(y;d,1)$ is the number of primes $p\equiv 1\pmod{d}$ up to $y$, and $\mathrm{li}y=\int_{2}^{y}\frac1{\log t} dt.$
\end{conjecture}
Our main result is stated as follows.
\begin{theorem}\label{thm1}
Assuming that the Elliott--Halberstam Conjecture is true, then we have
$$\sum_{n\leqslant x}\omega^*(n)^2\sim 2\frac{\zeta(2)\zeta(3)}{\zeta(6)}x\log x$$
as $x\rightarrow \infty$.
\end{theorem}

\section{Auxiliary results}
Throughout our paper, the number $x$ is supposed to be sufficiently large whereas $\epsilon$ is a given number which was supposed to be sufficiently small. The symbol $p$ will always be a prime.
In this section, we provide some preliminary results to be used later.

We need firstly the following asymptotic formulae of Landau \cite{La}.
\begin{lemma} \label{landau}
For  positive number $y$ we have
$$\sum_{m\leqslant y}\frac1{\varphi(m)}=\frac{\zeta(2)\zeta(3)}{\zeta(6)}\log y+O(1)$$
and
$$\sum_{m\leqslant y}\frac m{\varphi(m)}=\frac{\zeta(2)\zeta(3)}{\zeta(6)}y+O(\log y).$$
\end{lemma}

A classical inequality below involving primes in arithmetic progressions will be used frequently.

\begin{lemma} [Brun--Titchmarsh inequality \cite{MV}]\label{lem1}
Let $x$ be a positive real number, and let $k,a$ be relatively prime positive integers. Then
$$
\pi(x;k,a)\leqslant\frac{2x}{\varphi(k)\log(x/k)},
$$
provided that $x>k$.
\end{lemma}

The following result is obtained directly from Lemma \ref{lem1}.

\begin{lemma} \label{lem}
Let $x$ be a positive real number, and let $m,a$ be relatively prime positive integers. Then
$$\int_{2}^{x}\frac{\pi(t;m,1)}{(t-1)^2}dt\ll \frac{\log(\log x-\log m)}{\varphi(m)}$$
provided that $x>6m$, where the implied constant is absolute.
\end{lemma}
\begin{proof}
Trivially, we have
$$
\int_{2}^{x}\frac{\pi(t;m,1)}{(t-1)^2}dt\ll\int_{m}^{x}\frac{\pi(t;m,1)}{t^2}dt.
$$
We separate the integral above into the following three parts
$$\int_{m}^{x}\frac{\pi(t;m,1)}{t^2}dt=\left(\int_{m}^{em}+\int_{em}^{x/2}+\int_{x/2}^{x}\right)\frac{\pi(t;m,1)}{t^2}dt.$$
Using the Brun--Titchmarsh inequality for the second part as well as the trivial estimates for the other two parts, we deduce that
\begin{align}\label{eq3-4}
\int_{2}^{x}\frac{\pi(t;m,1)}{(t-1)^2}dt&\ll \int_{m}^{em}\frac{dt}{mt}+\int_{em}^{x/2}\frac{dt}{\varphi(m)t\log (t/m)}+\int_{x/2}^{x}\frac{dt}{mt}\nonumber\\
&\ll \frac{1}{m}+\frac{\log(\log x-\log m)}{\varphi(m)}\nonumber\\
&\ll \frac{\log(\log x-\log m)}{\varphi(m)}.
\end{align}

This completes the proof of Lemma \ref{lem}.
\end{proof}

The whole proof of our article is based on the following substantial identity of Murty and Murty \cite[equation (4.10)]{MM}. 
\begin{lemma}\label{lem3-1} For positive number $x$ we have
\begin{align*}
\sum_{n\leqslant x}\omega^*(n)^2=x\sum_{m\leqslant x}\varphi(m)\bigg(\sum_{\substack{p\leqslant x\\p\equiv 1\!\!\!\pmod{m}}}\frac{1}{p-1}\bigg)^2+O(x).
\end{align*}
\end{lemma}

The next lemma is deduced from partial summations and elementary estimates.
\begin{lemma}\label{lem3-3} For positive number $x\ge 2$ we have
\begin{align*}
\sum_{n\leqslant x}\omega^*(n)^2=x\sum_{1<m\leqslant x}\varphi(m)\bigg(\int_{2}^{x}\frac{\pi(t;m,1)}{(t-1)^2}dt\bigg)^2
+O\left(x(\log\log x)^3\right).
\end{align*}
\end{lemma}
\begin{proof}
From Lemma \ref{lem3-1}, we have
\begin{align}\label{eq3-1}
\sum_{n\leqslant x}\omega^*(n)^2=x\sum_{m\leqslant x}\varphi(m)\bigg(\sum_{\substack{p\leqslant x\\p\equiv 1\!\!\!\pmod{m}}}\frac{1}{p-1}\bigg)^2+O(x).
\end{align}
It is clear that the contribution from term $m=1$ is bounded by $O(x\log\log x)$.

Integrating by parts, we get
\begin{align}\label{eq3-2}
\sum_{\substack{p\leqslant x\\p\equiv 1\!\!\!\pmod{m}}}\frac{1}{p-1}=\frac{\pi(x;m,1)}{x-1}+\int_{2}^{x}\frac{\pi(t;m,1)}{(t-1)^2}dt.
\end{align}
Inserting Eq. (\ref{eq3-2}) into Eq. (\ref{eq3-1}), we deduce that
\begin{align}\label{eq3-3}
\sum_{n\le x}\omega^*(n)^2=S_1(x)+S_2(x)+S_3(x)+O\left(x\right),
\end{align}
where
$$S_1(x)=x\sum_{1<m\leqslant x}\varphi(m)\bigg(\int_{2}^{x}\frac{\pi(t;m,1)}{(t-1)^2}dt\bigg)^2,$$
$$S_2(x)=\frac{2x}{x-1}\sum_{m\leqslant x}\varphi(m)\pi(x;m,1)\int_{2}^{x}\frac{\pi(t;m,1)}{(t-1)^2}dt$$
and
$$S_3(x)=\frac{x}{(x-1)^2}\sum_{m\leqslant x}\varphi(m)\pi(x;m,1)^2.$$
 Then we will show that $S_2(x)$ and $S_3(x)$ provide the error terms.

We first deal with the sum $S_3(x)$ as it is easy to be bounded. By the Brun--Titchmarsh inequality (i.e., Lemma \ref{lem1}), we have
\begin{align*}
\frac{x}{(x-1)^2}\sum_{m\leqslant x/\log x}\varphi(m)\pi(x;m,1)^2&\ll x\sum_{m\leqslant x/\log x}\frac1{\varphi(m)}\frac{1}{\log^2 (x/m)}\\
&\ll x\sum_{m\leqslant x/\log x}\frac{\log\log m}{m}\frac{1}{\log^2 (x/m)}\\
&\ll x\log\log x,
\end{align*}
where, in the last but one estimation above, we employed the well--known estimate $\varphi(m)\gg m/\log\log m$.
Thus, by the trivial estimate we obtain
\begin{align}\label{S3}
S_3(x)&=\frac{x}{(x-1)^2}\sum_{x/\log x<m\leqslant x}\varphi(m)\pi(x;m,1)^2+O\left(x\log\log x\right)\nonumber\\
&\ll x\sum_{x/\log x<m\leqslant x}\frac{\varphi(m)}{m^2}+x\log x\log x
\nonumber\\
&\ll x\log\log x.
\end{align}

Next, we bound the sum $S_2(x)$. Since $\pi(t;m,1)=0$ for $t\leqslant m$, from the trivial estimate we have
\begin{align*}
\sum_{x/16<m\leqslant x}\varphi(m)\pi(x;m,1)\int_{2}^{x}\frac{\pi(t;m,1)}{(t-1)^2}dt&\ll x\sum_{x/16<m\leqslant x}\frac{\varphi(m)}{m}\int_{m}^{x}\frac{1}{mt}dt\\
&\leqslant x\sum_{x/16<m\leqslant x}\frac{\log x-\log m}{m}\\
&\ll x.
\end{align*}
We then prove that $$\sum_{m\leqslant x/16}\varphi(m)\pi(x;m,1)\int_{2}^{x}\frac{\pi(t;m,1)}{(t-1)^2}dt\ll x(\log\log x)^3.$$
From the estimate of Lemma \ref{lem} together with the Brun--Titchmarsh inequality, we deduce that
\begin{align*}
\sum_{m\leqslant x/16}\varphi(m)\pi(x;m,1)\int_{2}^{x}\frac{\pi(t;m,1)}{(t-1)^2}dt
&\ll \sum_{m\leqslant x/16}\varphi(m)\frac{x}{\varphi(m)\log (x/m)}\frac{\log\log x}{\varphi(m)}\\
&\ll x\log\log x\sum_{m\leqslant x/16}\frac{\log\log m}{m\log (x/m)}\\
&\ll x(\log\log x)^3,
\end{align*}
where, in the last but one estimate above, we used the estimate $\varphi(m)\gg m/\log\log m$ again. Thus, we proved that
\begin{align}\label{S2}
S_2(x)\ll x(\log\log x)^3.
\end{align}

Now, the lemma follows immediately from Eqs. (\ref{eq3-1}), (\ref{eq3-3}), (\ref{S3}) and (\ref{S2}).
\end{proof}

\section{Applications of the sieve results}

The aims of this section are the establishments of the following Propositions \ref{pro1} and \ref{pro2} which will be used later in the proof of Theorem \ref{thm1}.

\begin{proposition}\label{pro1}
Let $\epsilon$ be a sufficiently small positive number. Then
\begin{align*}
\sum_{x^{1-\epsilon}<m\leqslant x}\varphi(m)\bigg(\int_{2}^{x}\frac{\pi(t;m,1)}{(t-1)^2}dt\bigg)^2\ll \epsilon \log x.
\end{align*}
\end{proposition}

\begin{proposition}\label{pro2}
Let $\epsilon$ be a sufficiently small positive number. Then
\begin{align*}
\sum_{m\leqslant x^{1-\epsilon}}\varphi(m)\bigg(\int_{2}^{m^{1+\epsilon}}\frac{\pi(t;m,1)}{(t-1)^2}dt\bigg)^2\ll\epsilon \log x+\log\log x
\end{align*}
\end{proposition}

\begin{proposition}\label{pro3}
Let $\epsilon$ be a sufficiently small positive number. Then
\begin{align*}
\sum_{m\leqslant x^{1-\epsilon}}\varphi(m)\int_{2}^{m^{1+\epsilon}}\frac{\pi(t;m,1)}{(t-1)^2}dt\int_{m^{1+\epsilon}}^{x}\frac{\pi(t;m,1)}{(t-1)^2}dt
\ll\epsilon(\log \epsilon^{-1})^2\log x+(\log\log x)^2.
\end{align*}
\end{proposition}

In order to prove the propositions, we need some more lemmas. The following one is quoted from a recent article of Ding and Zhao.

\begin{lemma}\cite[Proposition 2.1]{Di-Zh}\label{Ding-lem}
For any given positive integers $g\geqslant 2$ and $\ell$, we have
$$\sum_{1<h_1<\cdot\cdot\cdot<h_g<z}\frac{1}{h_1\cdot\cdot\cdot h_g}\prod_{p|E_{h_1,...,h_g}}\left(1+\frac{1}{p}\right)^\ell\ll (\log z)^g,$$
provided that $z$ is sufficiently large, where
$$E_{h_1,...,h_g}=h_1\cdot\cdot\cdot h_g\prod_{1\leqslant i<j\leqslant g}(h_j-h_i)$$
and the implied constant depends only on $g$ and $\ell$.
\end{lemma}

We also need a weak form of the standard result \cite[Theorem 2.3, Page 70]{Halberstam} deduced from the Brun sieve.

\begin{lemma}\label{lem2}
Let $g$ be a natural number, and let $a_i,b_i~(i=1,...,g)$ be integers satisfying
$$(a_i,b_i)=1 \quad (i=1,...,g)$$
and
$$E:=\prod_{i=1}^ga_i\prod_{1\leqslant r<s\leqslant g}(a_rb_s-a_sb_r)\neq0.$$
Then
$$\#\{n\leqslant y:a_in+b_i\in \mathcal{P}, i=1,...,g\}\ll \prod_{p|E}\left(1-\frac1p\right)^{\rho(p)-g}\frac{y}{(\log y)^{g}},$$
where $\rho(p)$ denotes the number of solutions of
$$\prod_{i=1}^g(a_in+b_i)\equiv 0\pmod{p},$$
and the constant implied by the $\ll-$symbol depends on $g$ and $b_1,...,b_g$.
\end{lemma}

We are ready to offer the proof of Proposition \ref{pro1}.
\begin{proof}[Proof of Proposition \ref{pro1}]
Let
$$
\mathcal{A}_{1}(x)=\sum_{x^{1-\epsilon}<m\leqslant x}\varphi(m)\bigg(\sum_{\substack{p\leqslant x\\p\equiv 1\!\!\!\pmod{m}}}\frac{1}{p-1}\bigg)^2.
$$
Expanding the sum, we have
$$
\mathcal{A}_{1}(x)=\sum_{x^{1-\epsilon}<m\leqslant x}\varphi(m)\sum_{\substack{p_1,~p_2\leqslant x\\p_1\equiv 1\!\!\!\pmod{m}\\p_2\equiv 1\!\!\!\pmod{m}}}\frac{1}{(p_1-1)(p_2-1)}.
$$
Suppose that $p_i=m\ell_i+1~(i=1,2)$, then $\ell_i<x/m$. Changing orders of the sums above, we have
\begin{align*}
\mathcal{A}_{1}(x)&\leqslant \sum_{\ell_1,\ell_2<x^{\epsilon}}\frac{1}{\ell_1\ell_2}\sum_{\substack{x^{1-\epsilon}<m\leqslant x\\ m\ell_i+1\in \mathcal{P},~ i=1,2}}\frac{\varphi(m)}{m^2}
\leqslant \sum_{\ell_1,\ell_2<x^{\epsilon}}\frac{1}{\ell_1\ell_2}\sum_{\substack{x^{1-\epsilon}<m\leqslant x\\ m\ell_i+1\in \mathcal{P},~ i=1,2}}\frac{1}{m},
\end{align*}
where $\mathcal{P}$ denotes the set of primes.
Separating the case of $\ell_1=\ell_2$ with $\ell_1\neq \ell_2$, it follows that
\begin{align}\label{equa1}
\mathcal{A}_{1}(x)\ll \sum_{\ell_1<x^{\epsilon}}\frac{1}{\ell_1^2}\sum_{\substack{x^{1-\epsilon}<m\leqslant x\\ m\ell_1+1\in \mathcal{P}}}\frac{1}{m}+\sum_{\ell_1<\ell_2<x^{\epsilon}}\frac{1}{\ell_1\ell_2}\sum_{\substack{x^{1-\epsilon}<m\leqslant x\\ m\ell_i+1\in \mathcal{P},~ i=1,2}}\frac{1}{m}.
\end{align}

Trivially, we have
\begin{align}\label{eq-13-1}
\sum_{\ell_1<x^{\epsilon}}\frac{1}{\ell_1^2}\sum_{\substack{x^{1-\epsilon}<m\leqslant x\\ m\ell_1+1\in \mathcal{P}}}\frac{1}{m}\le \sum_{\ell_1<x^{\epsilon}}\frac{1}{\ell_1^2}\sum_{\substack{x^{1-\epsilon}<m\leqslant x}}\frac{1}{m}\ll \epsilon \log x.
\end{align}
Conveniently, we denote
$$
\mathcal{W}(t)=\sum_{\substack{m\leqslant t\\m\ell_i+1\in \mathcal{P},~ i=1,2}}1.
$$
Integrating by parts again, we find that
$$\sum_{\substack{x^{1-\epsilon}< m\leqslant x\\ m\ell_i+1\in \mathcal{P}, ~i=1,2}}\frac{1}{m}=\frac{\mathcal{W}(x)}{x}-\frac{\mathcal{W}(x^{1-\epsilon})}{x^{1-\epsilon}}
+\int_{x^{1-\epsilon}}^{x}\frac{\mathcal{W}(t)}{t^2}dt.$$
By Lemma \ref{lem2} with $g=2$, we have
\begin{align*}
\sum_{\substack{x^{1-\epsilon}< m\leqslant x\\ m\ell_i+1\in \mathcal{P}, ~i=1,2}}\frac{1}{m}&\ll \prod_{p|\ell_1\ell_2(\ell_2-\ell_1)}\left(1-\frac{1}{p}\right)^{-2}\Big(\frac{1}{\log^2 x}+\int_{x^{1-\epsilon}}^{x}\frac{1}{t(\log t)^2}dt\Big)\\
&\ll \prod_{p|\ell_1\ell_2(\ell_2-\ell_1)}\left(1+\frac{1}{p}\right)^2\frac{1}{\log x},
\end{align*}
which clearly means that
\begin{align}\label{equa3}
\sum_{\ell_1<\ell_2<x^{\epsilon}}\frac{1}{\ell_1\ell_2}\sum_{\substack{x^{1-\epsilon}<m\leqslant x\\ m\ell_i+1\in \mathcal{P}, ~i=1,2}}\frac{1}{m}\ll \frac{1}{\log x}\sum_{\ell_1<\ell_2<x^{\epsilon}}\frac{1}{\ell_1\ell_2}\prod_{p|\ell_1\ell_2(\ell_2-\ell_1)}\left(1+\frac{1}{p}\right)^2.
\end{align}
Employing Lemma \ref{Ding-lem} with $g=\ell=2$, we have
\begin{align}\label{equa4}
\sum_{\ell_1<\ell_2<x^{2\epsilon}}\frac{1}{\ell_1\ell_2}\prod_{p|\ell_1\ell_2(\ell_2-\ell_1)}\left(1+\frac{1}{p}\right)^2\ll \epsilon \log^2x
\end{align}
Thus, we obtain by Eqs. (\ref{equa1}), (\ref{eq-13-1}), (\ref{equa3}) and (\ref{equa4}) that
\begin{align*}
\mathcal{A}_{1}(x)\ll \epsilon \log x
\end{align*}

The sum to be estimated is clearly bounded by $\mathcal{A}_{1}(x)$ due to the partial summations
\begin{align}\label{eq-12-1}
\int_{2}^{x}\frac{\pi(t;m,1)}{(t-1)^2}dt\le \sum_{\substack{p\leqslant x\\p\equiv 1\!\!\!\pmod{m}}}\frac{1}{p-1},
\end{align}
which ends up the proof of our proposition.
\end{proof}

We next provide the proof of Proposition \ref{pro2}.
\begin{proof}[Proof of Proposition \ref{pro2}]
On noting Eq. (\ref{eq-12-1}), we easily find that the sum to be bounded is not larger than
\begin{align*}
\sum_{m\leqslant x^{1-\epsilon}}\varphi(m)\Bigg(\sum_{\substack{p\leqslant m^{1+2\epsilon}\\ p\equiv 1\!\!\!\pmod{m}}}\frac{1}{p-1}\Bigg)^2
:=\mathcal{A}_{2}(x),~\text{say}.
\end{align*}
By expanding the sums, it is clear that
\begin{align*}
\mathcal{A}_{2}(x)=\sum_{m\leqslant x^{1-\epsilon}}\varphi(m)\sum_{\substack{p_i\leqslant m^{1+2\epsilon}\\ p_i\equiv 1\!\!\!\pmod{m},~ i=1,2}}\frac{1}{(p_1-1)(p_2-1)}.
\end{align*}
Suppose that $p_i=m\ell_i+1~(i=1,2)$, then
$$
\ell_i<m^{2\epsilon}\leqslant x^{2\epsilon},
$$
and hence $m>\ell_i^{\frac{1}{2\epsilon}}$. Changing orders of the sums above, we have
\begin{align*}
\mathcal{A}_{2}(x) \leqslant \sum_{\ell_1,\ell_2\leqslant x^{2\epsilon}}\frac{1}{\ell_1\ell_2}\sum_{\substack{\ell_i^{\frac{1}{2\epsilon}}< m\leqslant x\\ m\ell_i+1\in \mathcal{P},~ i=1,2}}\frac{\varphi(m)}{m^2}
\leqslant 2\sum_{\ell_1\leqslant\ell_2\leqslant x^{2\epsilon}}\frac{1}{\ell_1\ell_2}\sum_{\substack{\ell_2^{\frac{1}{2\epsilon}}< m\leqslant x\\ m\ell_i+1\in \mathcal{P},~ i=1,2}}\frac{1}{m}.
\end{align*}
Separating the case of $\ell_1=\ell_2$ with $\ell_1\neq \ell_2$, it follows that
\begin{align}\label{inequality2}
\mathcal{A}_{2}(x)\ll \!\sum_{\ell_1\leqslant x^{2\epsilon}}\frac{1}{\ell_1^2}\sum_{\substack{\ell_1< m\leqslant x\\ m\ell_1+1\in \mathcal{P}}}\frac{1}{m}+\!\!\sum_{\substack{\ell_1<\ell_2\leqslant x^{2\epsilon}}}\!\!\frac{1}{\ell_1\ell_2}\sum_{\substack{\ell_2^{\frac{1}{2\epsilon}}< m\leqslant x\\ m\ell_i+1\in \mathcal{P},~ i=1,2}}\frac{1}{m}:=\mathcal{A}_{2,1}(x)+\mathcal{A}_{2,2}(x),
\end{align}
say.
It is plain that
\begin{align*}
\mathcal{A}_{2,1}(x)=\sum_{\ell_1\leqslant e}\frac{1}{\ell_1^2}\sum_{\substack{\ell_1<m\leqslant x\\ m\ell_1+1\in \mathcal{P}}}\frac{1}{m}+\sum_{e<\ell_1\leqslant x^{2\epsilon}}\frac{1}{\ell_1^2}\sum_{\substack{\ell_1<m\leqslant x\\ m\ell_1+1\in \mathcal{P}}}\frac{1}{m}
\end{align*}
and
$$
\sum_{\ell_1\leqslant e}\frac{1}{\ell_1^2}\sum_{\substack{m\leqslant x\\ m\ell_1+1\in \mathcal{P}}}\frac{1}{m}\le\sum_{\substack{m\leqslant x\\ m+1\in \mathcal{P}}}\frac{1}{m}+\sum_{\substack{m\leqslant x\\ 2m+1\in \mathcal{P}}}\frac{1}{m}<2\sum_{\substack{p\leqslant 2x+1,~ p\in \mathcal{P}}}\frac{3}{p}\ll \log\log x
$$
via the Mertens' formula.

Obviously, we are in a position to employ Lemma \ref{lem2}. Conveniently, we denote
$$
\mathcal{W}_1(t)=\sum_{\substack{m\leqslant t\\~m\ell_1+1\in \mathcal{P}}}1 \quad \text{and} \quad \mathcal{W}_2(t)=\sum_{\substack{m\leqslant t\\m\ell_i+1\in \mathcal{P},~ i=1,2}}1.
$$
Integrating by parts, we find (for $\ell_1>e$) that
$$\sum_{\substack{\ell_1<m\leqslant x\\ m\ell_1+1\in \mathcal{P}}}\frac{1}{m}=\frac{\mathcal{W}_1(x)}{x}-\frac{\mathcal{W}_1(\ell_1)}{\ell_1}
+\int_{\ell_1}^{x}\frac{\mathcal{W}_1(t)}{t^2}dt.$$
By Lemma \ref{lem2} with $g=1$, we have
\begin{align*}
\sum_{\substack{\ell_1<m\leqslant x\\ m\ell_1+1\in \mathcal{P}}}\frac{1}{m}&\ll \prod_{p|\ell_1}\left(1-\frac{1}{p}\right)^{-1}\left(\frac{1}{\log \ell_1}+\int_{e}^{x}\frac{1}{t\log t}dt\right)\\
&\ll \frac{\ell_1}{\varphi(\ell_1)}\left(\frac{1}{\log\ell_1}+\log\log x\right),
\end{align*}
which clearly means that
\begin{align}\label{inequality3}
\mathcal{A}_{2,1}(x)\ll\sum_{\ell_1<x^{2\epsilon}}\frac{1}{\ell_1^2}\sum_{\substack{\ell_1< m\leqslant x\\ m\ell_1+1\in \mathcal{P}}}\frac{1}{m}+\log\log x\ll \log\log x.
\end{align}

Similar to $\mathcal{A}_{2,1}(x)$, we have
\begin{align}\label{eq-15-1}
\mathcal{A}_{2,2}(x)=\sum_{\substack{ e<\ell_2\leqslant x^{2\epsilon}\\ \ell_1<\ell_2}}\frac{1}{\ell_1\ell_2}\sum_{\substack{\ell_2^{\frac{1}{2\epsilon}}< m\leqslant x\\ m\ell_i+1\in \mathcal{P},~ i=1,2}}\frac{1}{m}+O(\log\log x).
\end{align}
Integrating by parts again, we find (for $\ell_2>e$) that
$$\sum_{\substack{\ell_2^{\frac{1}{2\epsilon}}< m\leqslant x\\ m\ell_i+1\in \mathcal{P}, ~i=1,2}}\frac{1}{m}=\frac{\mathcal{W}_2(x)}{ x}-\frac{\mathcal{W}_2\left(\ell_2^{\frac{1}{2\epsilon}}\right)}{\ell_2^{\frac{1}{2\epsilon}}}
+\int_{\ell_2^{\frac{1}{2\epsilon}}}^{x}\frac{\mathcal{W}_2(t)}{t^2}dt.$$
By Lemma \ref{lem2} with $g=2$, we have
\begin{align}\label{inequality4}
\sum_{\substack{\ell_2^{\frac{1}{2\epsilon}}< m\leqslant x\\ m\ell_i+1\in \mathcal{P}, ~i=1,2}}\frac{1}{m}&\ll \prod_{p|\ell_1\ell_2(\ell_2-\ell_1)}\left(1-\frac{1}{p}\right)^{-2}\Big(\frac{1}{(\log x)^2}+\frac{\epsilon^2}{(\log \ell_2)^2}+\int_{\ell_2^{\frac{1}{2\epsilon}}}^{x}\frac{1}{t(\log t)^2}dt\Big)\nonumber\\
&\ll \prod_{p|\ell_1\ell_2(\ell_2-\ell_1)}\left(1+\frac{1}{p}\right)^2\left(\frac{1}{\log x}+\frac{\epsilon}{\log \ell_2}\right).
\end{align}
Inserting Eq. (\ref{inequality4}) into Eq. (\ref{eq-15-1}), we get
\begin{align}\label{eq-15-2}
\mathcal{A}_{2,2}(x)\ll \log\log x+\epsilon\log x+\epsilon\sum_{\substack{ e<\ell_2\leqslant x^{2\epsilon}\\ \ell_1\leqslant\ell_2}}\frac{1}{\ell_1\ell_2\log \ell_2}\prod_{p|\ell_1\ell_2(\ell_2-\ell_1)}\left(1+\frac{1}{p}\right)^2
\end{align}
via Lemma \ref{Ding-lem}. Then, we will give a more cautious calculation for the reminding sum above as follows:
\begin{align}\label{inequality5}
 \sum_{\substack{ e<\ell_2\leqslant x^{2\epsilon}\\ \ell_1<\ell_2}} \frac{\prod\limits_{p|\ell_1\ell_2(\ell_2-\ell_1)}(1+\frac{1}{p})^2}{\ell_1\ell_2\log\ell_2}&\ll \sum_{\substack{ e<\ell_2\leqslant x^{2\epsilon}\\ \ell_1<\ell_2}}\frac{\prod\limits_{p|\ell_1\ell_2(\ell_2-\ell_1)}(1+\frac{1}{p})^2}{\ell_1\ell_2}\int_{\ell_2}^{\infty}\frac{dt}{t(\log t)^2}\nonumber\\
 &\ll\int_{e}^{\infty}\sum_{\substack{1\leqslant\ell_1<\ell_2\leqslant \min\{x^{2\epsilon},t\}}}\frac{\prod\limits_{p|\ell_1\ell_2(\ell_2-\ell_1)}(1+\frac{1}{p})^2}{\ell_1\ell_2}\frac{dt}{t(\log t)^2}.
\end{align}
We now divide the integral above into the following two shorter ones, with the aid of Lemma \ref{Ding-lem} again, to get
\begin{align}\label{eq-16-1}
\int_{e}^{x^{2\epsilon}}\sum_{\substack{1\leqslant\ell_1<\ell_2\leqslant t}}\frac{\prod\limits_{p|\ell_1\ell_2(\ell_2-\ell_1)}(1+\frac{1}{p})^2}{\ell_1\ell_2}\frac{dt}{t(\log t)^2}
 \ll\int_{e}^{x^{2\epsilon}}\frac{dt}{t}
 \ll\epsilon\log x+1
\end{align}
and
\begin{align}\label{eq-16-2}
\int_{x^{2\epsilon}}^{\infty}\frac{dt}{t(\log t)^2}\sum_{\substack{1\leqslant\ell_1<\ell_2\leqslant x^{2\epsilon}}}\frac{\prod\limits_{p|\ell_1\ell_2(\ell_2-\ell_1)}(1+\frac{1}{p})^2}{\ell_1\ell_2}
&\ll(\epsilon\log x)^2\int_{x^{2\epsilon}}^{\infty}\frac{dt}{t(\log t)^2}\nonumber\\
&\ll\epsilon\log x.
\end{align}
Collecting together Eqs. from (\ref{eq-15-1}) to (\ref{eq-16-2}), we obtain the estimate
\begin{align}\label{eq-16-3}
\mathcal{A}_{2,2}(x)\ll \epsilon \log x+\log\log x.
\end{align}

Our proposition follows from Eqs. (\ref{inequality2}), (\ref{inequality3}) and (\ref{eq-16-3}).
\end{proof}

Below we give the proof of the last proposition.
\begin{proof}[Proof of Proposition \ref{pro3}]
Let
$$
\mathcal{A}_{3}(x)=\sum_{m\leqslant x^{1-\epsilon}}\varphi(m)\int_{2}^{m^{1+\epsilon}}\frac{\pi(t;m,1)}{(t-1)^2}dt\int_{m^{1+\epsilon}}^{x}\frac{\pi(t;m,1)}{(t-1)^2}dt
$$
Note firstly that
\begin{align*}
\int_{m^{1+\epsilon}}^{x}\frac{\pi(t;m,1)}{(t-1)^2}dt&\ll\frac{1}{\varphi(m)}\int_{m^{1+\epsilon}}^{x}\frac{t}{(t-1)^2\log \frac{t}{m}}dt\ll\frac{\log\log x-\log\log m+\log \epsilon^{-1}}{\varphi(m)}
\end{align*}
via Lemma \ref{lem1}. Inserting this estimate into $\mathcal{A}_{3}(x)$, we get
\begin{align}\label{eq-17-1}
\mathcal{A}_{3}(x)=O\left(\mathcal{A}_{3,1}(x)\right)+O\left(\log \epsilon^{-1}\mathcal{A}_{3,2}(x)\right),
\end{align}
where
$$
\mathcal{A}_{3,1}(x)\ll \sum_{m\leqslant x^{1-\epsilon}}(\log\log x-\log\log m)\int_{2}^{m^{1+\epsilon}}\frac{\pi(t;m,1)}{(t-1)^2}dt
$$
and
$$
\mathcal{A}_{3,2}(x)\ll \sum_{m\leqslant x^{1-\epsilon}}\int_{2}^{m^{1+\epsilon}}\frac{\pi(t;m,1)}{(t-1)^2}dt.
$$

By partial summations, we clearly have
\begin{align*}
\int_{2}^{m^{1+\epsilon}}\frac{\pi(t;m,1)}{(t-1)^2}dt\leqslant \sum_{\substack{p\leqslant m^{1+\epsilon}\\p\equiv 1\!\!\!\pmod{m}}}\frac{1}{p-1}.
\end{align*}
Hence, on putting $p-1=\ell m$ we obtain
\begin{align*}
\mathcal{A}_{3,2}(x)\ll \sum_{m\leqslant x^{1-\epsilon}}\sum_{\substack{p\leqslant m^{1+\epsilon}\\p\equiv 1\!\!\!\pmod{m}}}\frac{1}{p-1}=\sum_{m\leqslant x^{1-\epsilon}}\sum_{\substack{\ell< m^\epsilon\\ m\ell+1\in \mathcal{P}}}\frac{1}{m\ell}\le \sum_{\ell\leqslant x^\epsilon}\frac{1}{\ell}\sum_{\substack{\ell<m\leqslant x\\ m\ell+1\in \mathcal{P}}}\frac{1}{m}.
\end{align*}
For $\ell\leqslant e$, it is plain that
$$
\sum_{\ell\leqslant e}\frac{1}{\ell}\sum_{\substack{\ell<m\leqslant x\\ m\ell+1\in \mathcal{P}}}\frac{1}{m}\leqslant \sum_{\substack{m\leqslant x\\ m+1\in \mathcal{P}}}\frac{1}{m}+\sum_{\substack{m\leqslant x\\ 2m+1\in \mathcal{P}}}\frac{1}{m}<2\sum_{\substack{p\leqslant 2x+1,~ p\in \mathcal{P}}}\frac{3}{p}\ll \log\log x,
$$
from which it follows that
\begin{align}\label{eq-17-31}
\mathcal{A}_{3,2}(x)\ll  \sum_{e<\ell\leqslant x^\epsilon}\frac{1}{\ell}\sum_{\substack{\ell<m\leqslant x\\ m\ell+1\in \mathcal{P}}}\frac{1}{m}
+\log\log x.
\end{align}

Let
$$
\mathcal{W}_1(t)=\sum_{\substack{m\leqslant t\\~m\ell+1\in \mathcal{P}}}1.
$$
Integrating by parts, we have
\begin{align*}
\sum_{\substack{\ell<m\leqslant x\\ m\ell+1\in \mathcal{P}}}\frac{1}{m}&=\frac{\mathcal{W}_1(x)}{x}-\frac{\mathcal{W}_1\left(\ell\right)}{\ell}+\int_{\ell}^{x}\frac{\mathcal{W}_1(t)}{t^2}dt\\
&\ll \prod_{p|\ell}\left(1-\frac{1}{p}\right)^{-1}\left(\frac{1}{\log x}+\frac{1}{\log \ell}+\int_{\ell}^{x}\frac{1}{t\log t}dt\right)\\
&\ll \prod_{p|\ell}\left(1-\frac{1}{p}\right)^{-1}\left(\frac{1}{\log \ell}+\log\log x-\log\log\ell\right)
\end{align*}
via Lemma \ref{lem2} for $\ell>e$. Taking the estimate above into Eq. (\ref{eq-17-31}), we deduce from Lemma \ref{Ding-lem} that
\begin{align}\label{eq-17-32}
\mathcal{A}_{3,2}(x)\ll  \sum_{e<\ell\leqslant x^\epsilon}\frac{\log\log x-\log\log\ell}{\varphi(\ell)} +\log\log x+\log\epsilon^{-1}.
\end{align}
Again, integrating by parts, we would have
\begin{align*}
\sum_{e<\ell\leqslant x^\epsilon}\frac{\log\log x-\log\log\ell}{\varphi(\ell)}=K\left(x^\epsilon\right)(\log\log x-\log\log x^\epsilon)+\int_{e}^{x^\epsilon}\frac{K(t)}{t\log t}dt,
\end{align*}
where
$$
K(t)=\sum_{e<\ell\leqslant t}\frac{1}{\varphi(m)}.
$$
Since $K(t)\ll \log t$ from Lemma \ref{landau}, it follows immediately that
\begin{align}\label{eq-17-33}
\mathcal{A}_{3,2}(x)\ll  \epsilon\log\epsilon^{-1}\log x+\log\log x+\log\epsilon^{-1}.
\end{align}

Now, by trivial adjustments of the arguments to $\mathcal{A}_{3,2}(x)$, we can obtain that
\begin{align*}
\mathcal{A}_{3,1}(x)&\leqslant \sum_{m\leqslant x^{1-\epsilon}}(\log\log x-\log\log m)\sum_{\substack{p\leqslant m^{1+\epsilon}\\p\equiv 1\!\!\!\pmod{m}}}\frac{1}{p-1}\\
&\le \sum_{\ell\leqslant x^\epsilon}\frac{1}{\ell}\sum_{\substack{\ell<m\leqslant x\\ m\ell+1\in \mathcal{P}}}\frac{\log\log x-\log\log m}{m}\\
&\le \sum_{\ell\leqslant x^\epsilon}\frac{\log\log x-\log\log \ell}{\ell}\sum_{\substack{\ell<m\leqslant x\\ m\ell+1\in \mathcal{P}}}\frac{1}{m}\\
&\ll \sum_{e<\ell\leqslant x^\epsilon}\frac{\log\log x-\log\log \ell}{\ell}\sum_{\substack{\ell<m\leqslant x\\ m\ell+1\in \mathcal{P}}}\frac{1}{m}+(\log\log x)^2\\
&\ll \sum_{e<\ell\leqslant x^\epsilon}\frac{(\log\log x-\log\log \ell)^2}{\ell}+(\log\log x)^2+\left(\log\epsilon^{-1}\right)^2.
\end{align*}
Routine computations give that
$$
\sum_{e<\ell \leqslant x^\epsilon}\frac{(\log\log x-\log\log \ell)^2}{\ell}\ll \epsilon(\log \epsilon^{-1})^2\log x+(\log\log x)^2,
$$
and hence
\begin{align}\label{eq-17-34}
\mathcal{A}_{3,1}(x)\ll \epsilon(\log \epsilon^{-1})^2\log x+(\log\log x)^2+\left(\log\epsilon^{-1}\right)^2.
\end{align}

The proposition now follows immediately from Eqs. (\ref{eq-17-1}), (\ref{eq-17-33}) and (\ref{eq-17-34}).
\end{proof}

\section{Proof of Theorem \ref{thm1}}

Now, let's start the proof of Theorem \ref{thm1}.

\begin{proof}[Proof of Theorem \ref{thm1}]
By Lemma \ref{lem3-3} and Proposition \ref{pro1}, we have
\begin{align*}
\sum_{n\leqslant x}\omega^*(n)^2/x=\sum_{1<m\leqslant x^{1-\epsilon}}\varphi(m)\bigg(\int_{2}^{x}\frac{\pi(t;m,1)}{(t-1)^2}dt\bigg)^2
+O\left(\epsilon \log x+(\log\log x)^3\right).
\end{align*}
Note that
\begin{align*}
\int_{2}^{x}\frac{\pi(t;m,1)}{(t-1)^2}dt=\int_{2}^{m^{1+\epsilon}}\frac{\pi(t;m,1)}{(t-1)^2}dt+\int_{m^{1+\epsilon}}^{x}\frac{\pi(t;m,1)}{(t-1)^2}dt.
\end{align*}
Thus, we obtain
\begin{align*}
\sum_{1<m\leqslant x^{1-\epsilon}}\varphi(m)\bigg(\int_{2}^{x}\frac{\pi(t;m,1)}{(t-1)^2}dt\bigg)^2
=\sum_{1<m\leqslant x^{1-\epsilon}}\varphi(m)\bigg(\int_{m^{1+\epsilon}}^{x}\frac{\pi(t;m,1)}{(t-1)^2}dt\bigg)^2+\mathcal{R}_1(x)
\end{align*}
via Propositions \ref{pro2} and \ref{pro3}, where
\begin{align*}
\mathcal{R}_1(x)\ll\epsilon(\log \epsilon^{-1})^2\log x+(\log\log x)^3+\left(\log\epsilon^{-1}\right)^2.
\end{align*}
It follows that
\begin{align}\label{eq-19-1}
\sum_{n\leqslant x}\omega^*(n)^2/x=\sum_{1<m\leqslant x^{1-\epsilon}}\varphi(m)\bigg(\int_{m^{1+\epsilon}}^{x}\frac{\pi(t;m,1)}{(t-1)^2}dt\bigg)^2+\mathcal{R}_1(x).
\end{align}

On putting
\begin{align*}
\int_{m^{1+\epsilon}}^{x}\frac{\pi(t;m,1)}{(t-1)^2}dt=\int_{m^{1+\epsilon}}^{x}\frac{\pi(t)}{\varphi(m)(t-1)^2}dt+
\int_{m^{1+\epsilon}}^{x}\frac{\pi(t;m,1)-\pi(t)/\varphi(m)}{(t-1)^2}dt,
\end{align*}
we get from Eq. (\ref{eq-19-1}) that
\begin{align}\label{eq-19-2}
\sum_{n\leqslant x}\omega^*(n)^2/x=\sum_{1<m\leqslant x^{1-\epsilon}}\frac1{\varphi(m)}\bigg(\int_{m^{1+\epsilon}}^{x}\frac{\pi(t)}{(t-1)^2}dt\!\bigg)^2+\mathcal{R}_1(x)+\mathcal{R}_2(x)+\mathcal{R}_3(x),
\end{align}
where
$$
\mathcal{R}_2(x)\ll\sum_{1<m\leqslant x^{1-\epsilon}}\varphi(m)\bigg(\int_{m^{1+\epsilon}}^{x}\frac{\pi(t;m,1)-\pi(t)/\varphi(m)}{(t-1)^2}dt\bigg)^2
$$
and
$$
\mathcal{R}_3(x)\ll \sum_{1<m\leqslant x^{1-\epsilon}}\varphi(m)\int_{m^{1+\epsilon}}^{x}\frac{\pi(t)}{\varphi(m)(t-1)^2}dt
\int_{m^{1+\epsilon}}^{x}\frac{\left|\pi(t;m,1)-\pi(t)/\varphi(m)\right|}{(t-1)^2}dt.
$$

We are in a position to employ the Elliott--Halberstam Conjecture.
By Conjecture \ref{conjecture1} with $A=2$ and the trivial estimates, we have
\begin{align}\label{eq-19-3}
\mathcal{R}_3(x)
&\ll\sum_{1<m\leqslant x^{1-\epsilon}}\int_{m^{1+\epsilon}}^{x}\frac{1}{t\log t}dt
\int_{m^{1+\epsilon}}^{x}\left|\pi(t;m,1)-\frac{\pi(t)}{\varphi(m)}\right|\frac{1}{t^2}dt\nonumber\\
&\ll\log\log x\sum_{1<m\leqslant x^{1-\epsilon}}\int_{m^{1+\epsilon}}^{x}\left|\pi(t;m,1)-\frac{\pi(t)}{\varphi(m)}\right|\frac{1}{t^2}dt\nonumber\\
&\leqslant \log\log x\int_{2}^{x}\sum_{m\leqslant t^{\frac1{1+\epsilon}}}\left|\pi(t;m,1)-\frac{\pi(t)}{\varphi(m)}\right|\frac{1}{t^2}dt\nonumber\\
&\ll_{\epsilon} \log\log x\int_{2}^{x}\frac{1}{t(\log t)^2}dt\nonumber\\
&\ll_{\epsilon} \log\log x.
\end{align}
Note that
$$
\int_{m^{1+\epsilon}}^{x}\left|\pi(t;m,1)-\frac{\pi(t)}{\varphi(m)}\right|\frac{1}{(t-1)^2}dt\ll\frac{1}{\varphi(m)}\int_{m^{1+\epsilon}}^{x}\frac{1}{t\log t}dt\ll \frac{\log\log x}{\varphi(m)}
$$
Again, by Conjecture \ref{conjecture1} with $A=2$ and the trivial estimates, we have
\begin{align}\label{eq-19-4}
\mathcal{R}_3(x)&\ll \sum_{1<m\leqslant x^{1-\epsilon}}\varphi(m)\bigg(\int_{m^{1+\epsilon}}^{x}\frac{\pi(t;m,1)-\pi(t)/\varphi(m)}{(t-1)^2}dt\bigg)^2\nonumber\\
&\ll\log\log x\sum_{1<m\leqslant x^{1-\epsilon}}\int_{m^{1+\epsilon}}^{x}\left|\pi(t;m,1)-\frac{\pi(t)}{\varphi(m)}\right|\frac{1}{t^2}dt\nonumber\\
&\ll_{\epsilon} \log\log x
\end{align}
by the same argument with $\mathcal{R}_3(x)$.

Combing Eqs. (\ref{eq-19-1}), (\ref{eq-19-2}), (\ref{eq-19-3}) and (\ref{eq-19-4}), we conclude that
\begin{align}\label{eq-19-5}
\sum_{n\leqslant x}\omega^*(n)^2/x=\sum_{1<m\leqslant x^{1-\epsilon}}\frac1{\varphi(m)}\bigg(\!\int_{m^{1+\epsilon}}^{x}\frac{\pi(t)}{(t-1)^2}dt\bigg)^2+\mathcal{R}(x),
\end{align}
where
$$
\mathcal{R}(x)=O\left(\epsilon(\log \epsilon^{-1})^2\log x+(\log\log x)^3+\left(\log\epsilon^{-1}\right)^2\right)+O_\epsilon\left(\log\log x\right).
$$
Now, by the prime number theorem we have
\begin{align*}
\int_{m^{1+\epsilon}}^{x}\frac{\pi(t)}{(t-1)^2}dt&=(1+o(1))\int_{m^{1+\epsilon}}^{x}\frac{t}{(t-1)^2\log t}dt\\
&=(1+o(1))\int_{m^{1+\epsilon}}^{x}\frac{1}{t\log t}\left(1+O\left(\frac{1}{t}\right)\right)dt\\
&=(1+o(1))(\log\log x-\log\log m)+O\left(\epsilon+\frac{1}{m}\right).
\end{align*}
Taking the estimate above into Eq. (\ref{eq-19-5}), we deduce that
\begin{align}\label{eq-19-6}
\sum_{n\leqslant x}\omega^*(n)^2/x=(1+o(1))\sum_{1<m\leqslant x^{1-\epsilon}}\frac{(\log\log x-\log\log m)^2}{\varphi(m)}+\mathcal{R}(x)+\widetilde{\mathcal{R}}(x),
\end{align}
where
$$
\widetilde{\mathcal{R}}(x)\ll \epsilon\sum_{1<m\leqslant x}\frac{\log\log x-\log\log m}{\varphi(m)}+\epsilon \log x+\log\log x.
$$
Employing Lemma \ref{landau} with routine computations, we obtain
$$
\epsilon\sum_{1<m\leqslant x}\frac{\log\log x-\log\log m}{\varphi(m)}\ll \epsilon \log x
$$
and
$$
\sum_{1<m\leqslant x^{1-\epsilon}}\frac{(\log\log x-\log\log m)^2}{\varphi(m)}
=(1+O(\epsilon))2\frac{\zeta(2)\zeta(3)}{\zeta(6)}\log x+O\left((\log\log x)^3\right)
$$
by partial summations. Collecting the estimates above, we know from Eq. (\ref{eq-19-6}) that
\begin{align*}
\sum_{n \leqslant x}\omega^*(n)^2/x=f(\epsilon)\log x+O\left((\log\log x)^3+\left(\log\epsilon^{-1}\right)^2\right)+O_\epsilon(\log\log x),
\end{align*}
where
$$
f(\epsilon)=2\kappa+O(\epsilon(\log \epsilon^{-1})^2) \quad \text{and} \quad \kappa=\frac{\zeta(2)\zeta(3)}{\zeta(6)}.
$$
Hence, we have
$$
\sum_{n\leqslant x}\omega^*(n)^2/(x\log x)=f(\epsilon)+O\left(\frac{(\log\log x)^3}{\log x}+\frac{\left(\log \epsilon^{-1}\right)^2}{\log x}\right)+O_\epsilon\left(\frac{\log\log x}{\log x}\right),
$$
from which it follows that
\begin{align*}
2\kappa+O(\epsilon(\log \epsilon^{-1})^2)\leqslant \liminf_{x\rightarrow\infty}\frac{\sum\limits_{n\leqslant x}\omega^*(n)^2}{x\log x}\leqslant \limsup_{x\rightarrow\infty}\frac{\sum\limits_{n\leqslant x}\omega^*(n)^2}{x\log x}\leqslant 2\kappa+O(\epsilon(\log \epsilon^{-1})^2).
\end{align*}
Making $\epsilon\rightarrow 0$, we have
$$2\frac{\zeta(2)\zeta(3)}{\zeta(6)}\leqslant\liminf_{x\rightarrow\infty}\frac{\sum\limits_{n\leqslant x}\omega^*(n)^2}{x\log x}\leqslant \limsup_{x\rightarrow\infty}\frac{\sum\limits_{n\leqslant x}\omega^*(n)^2}{x\log x}\leqslant 2\frac{\zeta(2)\zeta(3)}{\zeta(6)}$$
since
$$
\epsilon(\log \epsilon^{-1})^2\rightarrow 0
$$
as $\epsilon\rightarrow 0$. Therefore,
$$\sum_{n\leqslant x}\omega^*(n)^2\sim 2\frac{\zeta(2)\zeta(3)}{\zeta(6)} x\log x, \quad (\text{as~} x\rightarrow\infty).$$

This completes the proof of Theorem \ref{thm1}.
\end{proof}

\section{Final remarks}
It is clear that we have the following corollary
$$\sum_{p,q\leqslant x}\frac{1}{[p-1,q-1]}\sim  2\frac{\zeta(2)\zeta(3)}{\zeta(6)}\log x$$
as $x\rightarrow \infty$ under the Elliott--Halberstam Conjecture due to (see \cite[page 6, last line]{MM})
$$\sum_{n\leqslant x}\omega^*(n)^2=\sum_{p,q\leqslant x}\frac{x}{[p-1,q-1]}+O(x)$$
and Theorem \ref{thm1}.

We cannot obtain explicit error terms in Theorem \ref{thm1} since it is certain that there is an oscillation of $\epsilon$--proportion of the main term which comes from the Elliott--Halberstam Conjecture naturally.

We believe that an unconditional proof of our main theorem is widely open and should be considered of independent interest. However, this seems unreachable under the current circumstance.

\section*{Acknowledgments}
The third named author would like to thank Lilu Zhao for his constant help and encouragement.

The first named author is supported by National Natural Science Foundation of China  (Grant No. 12201544), Natural Science Foundation of Jiangsu Province, China (Grant No. BK20210784), China Postdoctoral Science Foundation (Grant No. 2022M710121). He is also supported by  (Grant No. JSSCBS20211023 and YZLYJF2020PHD051).

The second named author is supported by the National Natural Science Foundation of China (No. 11901447), the Fundamental Research Funds for the Central Universities (No. xzy012021030) and the Shaanxi Fundamental Science Research Project for Mathematics and Physics (No. 22JSY006).

The third named author is supported by the National Key Research and Development Program of China (Grant No. 2021YFA1000700).

\end{document}